\documentclass[a4paper,reqno]{amsart}
\usepackage{amsmath,amssymb,amsthm,graphicx,mathrsfs,url}
\usepackage[usenames,dvipsnames]{color}
\definecolor{darkred}{rgb}{0.4,0.1,0.1}
\definecolor{darkblue}{rgb}{0.1,0.1,0.4}

\usepackage{tikz}
\usetikzlibrary{datavisualization}
\usetikzlibrary{datavisualization.formats.functions}
\usepackage[normalem]{ulem}

\numberwithin{equation}{section}
\theoremstyle{plain}% default

\newtheorem{theorem}{Theorem}[section]
\newtheorem{lemma}[theorem]{Lemma}

\newtheorem{lm}[theorem]{Lemma}
\newtheorem{proposition}[theorem]{Proposition}
\newtheorem{corollary}[theorem]{Corollary}

\theoremstyle{remark}
\newtheorem{remark}[theorem]{Remark}

\theoremstyle{definition}
\newtheorem{example}[theorem]{Example}

\newtheorem{assumption}[theorem]{Assumption}

\newtheorem{conjecture}[theorem]{Conjecture}

\newcommand\cD{\mathcal D}

\newcommand\cP{\mathcal P}

\newcommand\eps{\varepsilon}

\definecolor{darkgreen}{rgb}{0.1,0.45,0.1}
\definecolor{darkblue}{rgb}{0.1,0.1,0.4}
\definecolor{darkgrey}{rgb}{0.5,0.5,0.5}

\definecolor{darkred}{rgb}{0.6,0.0,0.0}

\newcommand\void[1]{}

\def\eps{\varepsilon}

\def\st{\mathfrak t}

\def\cD{{\mathcal D}}

\def\cP{{\mathcal P}}   \def\cQ{{\mathcal Q}}   
   \def\cT{{\mathcal T}}

\def\R{\mathbb{R}}

\def\N{\mathbb{N}}

\renewcommand{\div}{\mathrm{div}\,}

\newcommand{\dom}{\mathrm{dom}\,}

\allowdisplaybreaks

\def\dd{{\,\mathrm d}}

\newcounter{counter_a}

\title[Estimates for low Neumann eigenvalues]{Estimates for the lowest Neumann eigenvalues of parallelograms and domains of constant width}

\date{\today} 

\author[C.~L\'ena]{Corentin L\'ena}
\address{Università degli Studi di Padova - DTG\\
Stradella S. Nicola 3\\
 36100 Vicenza  \\
 Italy
 \newline
\indent Università degli Studi di Padova - Dipartimento di Matematica ``Tullio Levi-Civita''\\
via Trieste 63\\
35121 Padova\\
Italy}
 \email{corentin.lena@unipd.it}

\author[J.~Rohleder]{Jonathan Rohleder}
 \address{Matematiska institutionen \\ Stockholms universitet \\
 106 91 Stockholm \\
 Sweden}
 \email{jonathan.rohleder@math.su.se}
 


\begin{document}

\begin{abstract}
We prove sharp upper bounds for the first and second non-trivial eigenvalues of the Neumann Laplacian in two classes of domains: parallelograms and domains of constant width. This gives in particular a new proof of an isoperimetric inequality for parallelograms recently obtained by A.~Henrot, A.~Lemenant and I.~Lucardesi.
\end{abstract}

\maketitle

\section{Introduction}

We are concerned in this article with planar domains $\Omega$, that is, open, bounded and connected subsets of $\R^2$. We always assume that $\Omega$ is a Lipschitz domain, and we consider the sequence 
\begin{equation*}
	0=\mu_1(\Omega)<\mu_2(\Omega)\le\mu_3(\Omega)\le\dots,
\end{equation*}
consisting of the eigenvalues for the Neumann Laplacian, counted with multiplicities. We recall that the corresponding eigenvalue problem is 
\begin{equation*}
	\left\{\begin{array}{ll}
					-\Delta u=\mu u &\mbox{in }\Omega,\\
				    \frac{\partial u}{\partial \nu}=0&\mbox{on }\partial \Omega,
				\end{array}
	\right.
\end{equation*}
where $\frac{\partial u}{\partial \nu}$ denotes the outward-pointing normal derivative; in general, the derivative on $\partial \Omega$ in the direction of $\nu$ is defined in a weak sense, see Section \ref{sec:prel}.

It was  proved in 1954 by G.~Szeg\H{o} \cite{S54} that, among all simply-connected domains of a given area, the disk is the unique maximizer of $\mu_2(\Omega)$. Equivalently, for simply- connected domains $\Omega\subset\R^2$,
\begin{equation}
	\label{eqIneqSW}
	\mu_2(\Omega)|\Omega|\le \mu_2(\mathbb D)\, \pi,\end{equation}
where $|\Omega|$ denotes the area of $\Omega$ and $\mathbb D$ the unit disk in $\R^2$. We can note that the expression on the left-hand side of \eqref{eqIneqSW} is invariant under scaling of $\Omega$. Inequality \eqref{eqIneqSW} was extended to domains in any dimension, without the assumption of simple connectedness, by H.\,F.~Weinberger in 1956 \cite{W56}. Equality in \eqref{eqIneqSW} is attained only for disks (in higher dimension, only for balls).

Following R.\,L.~Laugesen and B.~Siudeja \cite{LS09}, we investigate how large $\mu_2(\Omega)$ can be when the perimeter is fixed, rather than the area. Equivalently, we look for upper bounds of the product
\[L(\Omega)^2\,\mu_2(\Omega),\]
where $L(\Omega)$ denotes the perimeter of $\Omega$ (this product is also scaling-invariant). As stressed by Laugesen and Siudeja in Problem~9.2 of the paper \cite{LS09} and by Laugesen in Problem~3, page 405 of the proceedings \cite{WBAL09}, this product is not maximized by disks. Indeed, the known formulas for the Neumann eigenvalues of the unit disk $\mathbb D$ give
\begin{equation*}
	4\pi^2\,\mu_2(\mathbb D)<16\pi^2,
\end{equation*}
while
\begin{equation*}
	L(\Omega)^2\,\mu_2(\Omega)=16\pi^2
\end{equation*}
when $\Omega$ is either a square or an equilateral triangle. In addition, Laugesen and Siudeja proved that 
\begin{equation*}
	L(\cT)^2\mu_2(\cT)\le16\pi^2,
\end{equation*}
for any triangle $\cT$, with equality only when $\cT$ is equilateral \cite[Theorem 3.1]{LS09}.  The question asked  by Laugesen  in \cite[p. 405, Problem 3]{WBAL09} immediately suggests the following conjecture.
\begin{conjecture} \label{conj}
For any convex domain $\Omega$ in $\R^2$, 
		\begin{equation*}
				L(\Omega)^2\,\mu_2(\Omega)\le 16\pi^2,
		\end{equation*}
and equality is attained only for squares and equilateral triangles. 
\end{conjecture} 
Besides the work of Laugesen and Siudeja mentioned above, Conjecture \ref{conj} was verified by A.~Raiko for parallelograms subject to certain geometrical restrictions \cite[Theorem 2]{R18}. As shown in \cite[Section 3]{BBG16}, \cite[Proposition 3.3]{HLL22}, \cite[Problem 9.2]{LS09} or \cite{P08}, the convexity hypothesis cannot be removed: one can construct sequences $(\Omega_n)$ of non-convex domains such that  
\[L(\Omega_n)^2\mu_2(\Omega_n)\to+\infty.\]

Motivated by Conjecture \ref{conj}, we find geometric upper bounds of $\mu_2(\Omega)$ and $\mu_3(\Omega)$, for two classes of domains $\Omega$. The first consists of all parallelograms. The second consists, in a certain sense, of domains of constant width, some of which are neither polygonal nor convex. We show that in these classes only the squares, respectively the rectangles, realize equality. As a corollary, we verify Conjecture \ref{conj} for all parallelograms, namely, the product $L(\cP)^2\,\mu_2(\cP)$, with $\cP$ a parallelogram, is maximized only by squares (see Theorem \ref{thm:parallelIsoperimetric} below). 

Our proofs use  Rayleigh's principle, with trial functions constructed from a suitable mapping of the domain onto the unit square. We introduce the necessary tools in Section \ref{sec:prel}, then study parallelograms in Section \ref{sec:parallel} (Theorem \ref{thm:parallel}) and domains of constant width in Section \ref{sec:constantWidth} (Theorem \ref{thm:constantWidth}).   Finally, we sketch in Section \ref{sec:perturbation} a simple perturbation argument that shows the existence of non-convex domains $\Omega$, close to the unit square, satisfying $L(\Omega)^2\mu_2(\Omega)>16\pi^2$.

During the preparation of this manuscript, we became aware of the recent work by A.~Henrot, A.~Lemenant and  I.~Lucardesi \cite{HLL22}. The authors prove the existence of a maximizer in the class of convex domains \cite[Proposition 3.1]{HLL22}. They also verify Conjecture \ref{conj} for all convex domains having two axes of symmetry (not necessarily perpendicular) \cite[Theorem 1.2]{HLL22} and,  as in our Theorem \ref{thm:parallelIsoperimetric}, for all parallelograms \cite[Proposition 4.3]{HLL22}. However, our work uses a different method and leads to new explicit estimates for $\mu_2 (\Omega)$ and $\mu_3(\Omega)$ which, as far as we can tell, cannot be directly deduced from \cite{HLL22}.

\section{Preliminaries}\label{sec:prel}

During the whole article, $\Omega \subset \R^2$ is a bounded, connected Lipschitz domain. The main object of our interest is the Laplacian $- \Delta_{\rm N}$ on $\Omega$ with Neumann boundary conditions. This self-adjoint, non-negative operator can be defined via its quadratic form
\begin{align*}
 H^1 (\Omega) \ni u \mapsto \int_\Omega |\nabla u|^2.
\end{align*}
Its domain consists of all $u \in H^1 (\Omega)$ such that $\Delta u$, taken distributionally, belongs to $L^2 (\Omega)$ and $u$ satisfies the boundary condition $\frac{\partial u}{\partial \nu} |_{\partial \Omega} = 0$ in a weak sense. We denote by
\begin{align*}
 0 = \mu_1 (\Omega) < \mu_2 (\Omega) \leq \dots
\end{align*}
the eigenvalues of $- \Delta_{\rm N}$, counted with multiplicities.

In the course of our investigations, certain auxiliary second-order elliptic differential operators with a weight function will play a role, which we define now. We point out that this approach is not new and appears, e.g., in work by  V.~Gol'dshtein, V.~Pchelintsev and A.~Ukhlov \cite{GU2019,GPU2021} in a much more general setting using the theory of weak quasiconformal mappings and the characterization of composition operators on Sobolev spaces.

For $w : \Omega \to (0, \infty)$ measurable, the space $L^2_w (\Omega)$ consists of all measurable $u : \Omega \to \R$ such that
\begin{align*}
 \|u\|_w^2 := \int_\Omega w (x, y) |u (x, y)|^2 \dd (x, y) < \infty.
\end{align*}
Then $\| \cdot \|_w$ defines a norm, with which $L^2_w (\Omega)$ is a Hilbert space, and we denote by $\langle \cdot, \cdot \rangle_w$ the corresponding inner product. In the rest of this section, let $f:\Omega\to\mathbb R$ be a measurable function such that
\begin{equation*}
 c\le f(x,y)\le C
\end{equation*}
for some fixed constants $0<c\le C$ and for  almost every $(x,y)\in \Omega$.  Moreover, we denote by $A : \overline{\Omega} \to \R^{2 \times 2}$ a continuous matrix function such that $A (x, y)^\top = A (x, y)$ and $A (x, y)$ is a positive  definite matrix for all $(x, y) \in \overline{\Omega}$.

\begin{proposition}
With the above hypotheses, the quadratic form $\st_{A, \Omega}$ in $L^2_{1/f} (\Omega)$ given by
\begin{align*}
 \st_{A, \Omega} [u, v] = \int_{\Omega} \langle A (x, y) \nabla u (x, y), \nabla v (x, y) \rangle \dd (x, y)
\end{align*}
with
\begin{align*}
 \dom \st_{A, \Omega} =H^1(\Omega)
\end{align*}
is symmetric, non-negative (hence semi-bounded below) and closed.
\end{proposition}

The symmetry and the non-negativity of $\st_{A, \Omega}$ follow from that of $A(x,y)$. The fact that $\st_{A, \Omega}$ is closed is an immediate consequence of the following lemma, which one can easily deduce from the hypotheses on $f$ and $A$. 

\begin{lm}\label{lmEqu} The norm associated with the form $\st_{A,\Omega}$, defined for $u\in H^1(\Omega)$ by
\[\|u\|_{A,\Omega}^2:=\st_{A,\Omega}[u,u] +\|u\|_{1/f}^2,\]
is equivalent to the  norm of $H^1 (\Omega)$.
\end{lm}

As described in \cite[Theorem VIII.15]{RS1}, we can associate with the form $\st_{A,\Omega}$ a self-adjoint operator $T_{A, \Omega}$, formally given by 
\begin{align*}
 T_{A, \Omega} u = - f (\cdot) \,\div (A (\cdot) \nabla u) .
\end{align*}
More precisely, we define 
\begin{equation*}
 \dom T_{A, \Omega}= 
  \left\{u\in H^1(\Omega): \exists   v\in L^2_{1/f}(\Omega) 
   ,\, \forall\varphi\in H^1(\Omega),\, t_{A,\Omega}[u,\varphi]=\langle v,\varphi\rangle_{1/f}
\right\},
 \end{equation*}
 and set $T_{A,\Omega}u=v$ for $u\in \dom T_{A, \Omega}$.
 
 In order to give a more concrete description of $\dom T_{A, \Omega}$, we note that for any $u\in H^1(\Omega)$, the mapping 
 \[\varphi\mapsto \st_{A,\Omega}[u,\varphi],\]
restricted to $\varphi\in C^\infty_c(\Omega)$, defines a distribution in $\cD'(\Omega)$ which we denote by $Pu$. Moreover, it follows from Lemma \ref{lmEqu} that $Pu$ belongs to the dual of $H^1(\Omega)$. From this and \cite[Lemma 4.3]{McL}, there exists an element $\gamma_1u$ in $H^{-1/2}(\partial\Omega)$ such that, for all $\varphi\in H^1(\Omega)$,

\begin{equation}\label{eqIBP}
	\langle- Pu ,\varphi\rangle_{\left(H^1(\Omega)\right)'\times H^{1}(\Omega)}+\st_{A,\Omega}[u,\varphi]=\langle \gamma_1 u ,\gamma_0 \varphi\rangle_{H^{-1/2}(\partial\Omega)\times H^{1/2}(\partial\Omega)}.
\end{equation} 
In the previous formula,
\begin{enumerate}
\item $\gamma_0:H^1(\Omega)\to H^{1/2}(\partial \Omega)$ is the usual boundary trace operator,
\item $\langle \zeta,z\rangle_{Z'\times Z}$ denotes the image of an element $z$ in a normed space $Z$ by an element $\zeta$ in the dual $Z'$ (note that $H^{-1/2}(\partial\Omega)=(H^{1/2}(\partial\Omega))')$.
\end{enumerate}
 To understand the meaning of  $\gamma_1u$, let us assume for a moment that $u\in C^1(\overline{\Omega})$. Then $\gamma_1u$ is the co-normal derivative, defined on $\partial\Omega$ by 
 \[(s,t)\mapsto \langle A(s,t)\nabla u(s,t),\nu(s,t)\rangle,\]
where $\nu(s,t)$ denotes the outward-pointing normal unit vector on $\partial \Omega$. We can therefore see $\gamma_1u$ as a generalized co-normal derivative and \eqref{eqIBP} as a generalized Green formula.

The following can be obtained by a standard argument.

\begin{proposition} The self-adjoint operator $T_{A,\Omega}$ can alternatively be defined by
 \[\dom T_{A, \Omega}=  \left\{u\in H^1(\Omega): Pu\in L^2(\Omega)\mbox{ and } \gamma_1u=0\right\}\]
 and
 \[ T_{A, \Omega}u=f\,Pu.\]
\end{proposition}

Next we consider instances of the operator $T_{A, \Omega}$ that arise from a diffeomorphic transformation of the Neumann Laplacian. The following lemma and Corollary~\ref{cor:Rayleigh} below can be found in greater generality in, e.g., \cite[Section~3]{GPU2021} and also in earlier works for Lipschitz transformations (see e.g.\ the review by V.\,I.~Burenkov, P.\,D.~Lamberti and M.~Lanza de Cristoforis \cite{BLLdC06} and references therein). However, to be self-contained we provide a short proof for our setting.

\begin{lemma}\label{lem:trafo}
Let $\Omega, \Omega' \subset \R^2$ be two bounded, connected Lipschitz domains such that there exists a $C^1$-diffeomorphism $\Phi$ which maps $\Omega$ onto $\Omega'$ such that both $\Phi$ and $\Phi^{-1}$ have bounded partial derivatives of order one. Let
\begin{align}\label{eq:A}
 A (s, t) = \left[\frac{1}{\big| \det D \Phi \big|} (D \Phi) (D \Phi)^\top \right] (\Phi^{-1} (s, t)), \quad (s, t) \in \Omega',
\end{align}
where $D \Phi$ denotes the Jacobi matrix of $\Phi$. Then a function $u$ belongs to $H^1 (\Omega')$ if and only if $u \circ \Phi$ belongs to $H^1 (\Omega)$, and in this case
\begin{align}\label{eq:forms}
 \int_{\Omega} |\nabla (u \circ \Phi) (x, y) |^2 \dd (x, y) = \int_{\Omega'} \langle A (s, t) \nabla u (s, t), \nabla u (s, t) \rangle \dd (s, t).
\end{align}
In particular, if we set $f (s, t) = |\det (D \Phi) (\Phi^{-1} (s, t))|$, then the Laplacian $- \Delta_{\rm N}$ in $L^2 (\Omega)$ with Neumann boundary conditions is isomorphic to the operator $T_{A,\Omega'}$ in $L^2_{1/f} (\Omega')$, and their spectra coincide.
\end{lemma}

\begin{proof}
Let $u \in H^1 (\Omega')$. As $\Phi$ maps $\Omega$ onto the bounded domain $\Omega'$, $\Phi$ is bounded. Moreover, by assumption, $\Phi$ has bounded partial derivatives. Hence $u \circ \Phi \in H^1 (\Omega)$. Moreover,
\begin{align*}
 \int_{\Omega} |\nabla (u \circ \Phi) (x, y) |^2 \dd (x, y) & = \int_\Omega |(D \Phi)^\top (x, y) (\nabla u) (\Phi (x, y))|^2 \dd (x, y) \\
 & = \int_{\Omega'} \frac{1}{|\det D \Phi|} |(D \Phi)^\top (\Phi^{-1} (s, t)) (\nabla u) (s, t)|^2 \dd (s, t) \\
 & = \int_{\Omega'} \langle A (s, t) \nabla u (s, t), \nabla u (s, t) \rangle \dd (s, t).
\end{align*}
Conversely, by analogous reasoning, for $v \in H^1 (\Omega)$ the function $u = v \circ \Phi^{-1}$ belongs to $H^1 (\Omega')$. In particular, the mapping $H^1 (\Omega') \ni u \mapsto u \circ \Phi \in H^1 (\Omega)$ provides an isomorphism between the quadratic forms corresponding to the operators $T_{A, \Omega'}$ in $L^2_{1/f} (\Omega')$ and $- \Delta_{\rm N}$ in $L^2 (\Omega)$. Hence, the two operators are isomorphic. 
\end{proof}

A for us important consequence of Lemma \ref{lem:trafo} is the following.

\begin{corollary}\label{cor:Rayleigh}
Let $\Omega, \Omega' \subset \R^2$ be two bounded, connected Lipschitz domains such that there exists a $C^1$-diffeomorphism $\Phi$ which maps $\Omega$ onto $\Omega'$ such that both $\Phi$ and $\Phi^{-1}$ have bounded partial derivatives of order one. Moreover, assume that $A$ and $f$ are defined as in Lemma \ref{lem:trafo}. Then
\begin{align*}
 \mu_k (\Omega) = \min_{\substack{F \subset H^1 (\Omega') \\ \dim F = k}} \max_{\substack{u \in F \\ u \neq 0}} \frac{\int_{\Omega'} \langle A (s, t) \nabla u (s, t), \nabla u (s, t) \rangle \dd (s, t)}{\int_{\Omega'} \frac{1}{f (s, t)} |u (s, t)|^2 \dd (s, t)}
\end{align*}
holds for all $k \in \N$. In particular,
\begin{align}\label{eq:varTransformed}
 \mu_2 (\Omega) = \min_{\substack{u \in H^1 (\Omega') \setminus \{0\} \\ \int_{\Omega'} \frac{1}{f} u = 0}} \frac{\int_{\Omega'} \langle A (s, t) \nabla u (s, t), \nabla u (s, t) \rangle \dd (s, t)}{\int_{\Omega'} \frac{1}{f (s, t)} |u (s, t)|^2 \dd (s, t)}.
\end{align}
Moreover, a non-trivial function $u \in H^1 (\Omega')$ which satisfies $\int_{\Omega'} \frac{1}{f} u = 0$ is a minimizer of \eqref{eq:varTransformed} if and only if $u \in \ker (T_{A, f} - \mu_2 (\Omega))$.
\end{corollary}

\section{Bounds for low eigenvalues of parallelograms}\label{sec:parallel}

In this section we derive eigenvalue bounds for the lowest non-zero eigenvalues $\mu_2 (\cP)$ and $\mu_3 (\cP)$ of the Neumann Laplacian on any parallelogram $\cP$. Our main result are the following sharp estimates.

\begin{theorem}\label{thm:parallel}
Let $\cP \subset \R^2$ be any parallelogram with side lengths $\ell_1, \ell_2$, area $|\cP|$ and one angle $\varphi$. Without loss of generality, let us assume $\ell_1 \leq \ell_2$. Define
\begin{align*}
 \lambda_\pm = \frac{\pi^2}{2 |\cP|^2} \left( \ell_1^2 + \ell_2^2 \pm \sqrt{\big(\ell_1^2 - \ell_2^2 \big)^2 + \frac{256}{\pi^4} \ell_1^2 \ell_2^2 \cos^2 \varphi} \right)
\end{align*}
and
\begin{align*}
 \eta_\pm = \frac{6}{|\cP|^2} \left( \ell_1^2 + \ell_2^2 \pm \sqrt{\big(\ell_1^2 - \ell_2^2 \big)^2 + 4 \ell_1^2 \ell_2^2 \cos^2 \varphi} \right).
\end{align*}
Then
\begin{align}\label{eq:mu2}
 \mu_2 (\cP) \leq \min \{\lambda_-, \eta_-\}
\end{align}
and 
\begin{align}\label{eq:mu3}
\mu_3 (\cP) \leq \lambda_+. 
\end{align}
In particular,
\begin{align}\label{eq:mean}
 \frac{\mu_2 (\cP) + \mu_3 (\cP)}{2} \leq \frac{\pi^2}{|\cP|^2} \frac{\ell_1^2 + \ell_2^2}{2}.
\end{align}
In \eqref{eq:mu2}  equality holds if, and only if,  $\cP$ is a rectangle, in which case 
\[\mu_2 (\cP)=\lambda_-=\frac{\pi^2}{\ell_2^2}.\]
In \eqref{eq:mu3} and \eqref{eq:mean} equality holds if, and only if, $\cP$ is a rectangle and $\ell_2 \le 2 \ell_1$. In this case, 
\[\mu_2(\cP)=\lambda_-=\frac{\pi^2}{\ell_2^2}\mbox{\ \ \ and \ \ \ }\mu_3(\cP)=\lambda_+=\frac{\pi^2}{\ell_1^2}.\]
\end{theorem}

\begin{remark} \label{remEquality} We can immediately make an observation that will be useful to study the case of equality. If $\lambda_+=\lambda_-$, an examination of the formulas reveals that we must have  $\ell_1=\ell_2$ and $\cos \varphi=0$, so that $\cP$ is a square and $\mu_2(\cP)=\lambda_-=\mu_3(\cP)=\lambda_+=\pi^2/\ell^2$, where $\ell$ is the length of an arbitrary side. Conversely, if $\cP$ is a square of side-length $\ell$, $\mu_2(\cP)=\lambda_-=\mu_3(\cP)=\lambda_+=\pi^2/\ell^2$. 
\end{remark}

\begin{proof}
Let $\cP \subset \R^2$ be the parallelogram spanned by the vectors $(a, b)^\top$ and $(c, d)^\top$, and let $\ell_1 = \sqrt{a^2 + b^2}$ and $\ell_2 = \sqrt{c^2 + d^2}$ be its side lengths; without loss of generality, $\ell_1 \leq \ell_2$. Then the linear transformation given by
\begin{align*}
 \Phi (x, y) = \frac{1}{a d - b c} \begin{pmatrix} d & - c \\ - b & a \end{pmatrix} \binom{x}{y}
\end{align*}
maps $\cP$ onto the unit square $\cQ := (0, 1)^2$. The constant matrix $A$ associated with $\Phi$ in \eqref{eq:A} is then given by
\begin{align}\label{eq:Aparallelogram}
 A = \frac{1}{|a d - b c|} \begin{pmatrix} c^2 + d^2 & - a c - b d \\ - a c - b d & a^2 + b^2 \end{pmatrix} = \frac{1}{|\cP|} \begin{pmatrix} \ell_2^2 & - \ell_1 \ell_2 \cos \varphi \\ - \ell_1 \ell_2 \cos \varphi & \ell_1^2 \end{pmatrix},
\end{align}
where $\varphi$ is the angle of $(a, b)^\top$ towards $(c, d)^\top$ and $|\cP|$ denotes the area of $\cP$. For using Corollary \ref{cor:Rayleigh}, note that, in the notation of the corollary, $f (s, t) = |a d - b c|^{-1} = |\cP|^{-1}$ constantly. 

In order to obtain eigenvalue estimates, we use two types of trial functions. Firstly, we consider the function
\begin{align}\label{eq:cosTest}
 u (s, t) = \alpha \cos (\pi s) + \beta \cos (\pi t), \quad (s, t) \in \cQ,
\end{align}
where $\alpha$ and $\beta$ are arbitrary real numbers. Note that $u$ is an eigenfunction of the Neumann Laplacian on $\cQ$ corresponding to $\mu_2 (\cQ) = \pi^2$ (in particular, $\int_\cQ u = 0 = \int_\cQ \frac{1}{f} u$) and that
\begin{align*}
 \nabla u (s, t) = - \pi \begin{pmatrix} \alpha \sin (\pi s) \\ \beta \sin (\pi t) \end{pmatrix}.
\end{align*}
Then 
\begin{align}\label{eq:formComp}
\begin{split}
 \int_{\cQ} \langle A \nabla u, \nabla u \rangle & = \frac{1}{|\cP|} \int_{\cQ} \Big( \ell_2^2 (\partial_1 u)^2 - 2 \ell_1 \ell_2 \cos \varphi (\partial_1 u) (\partial_2 u) + \ell_1^2 (\partial_2 u)^2 \Big) \\
 & = \frac{\pi^2}{|\cP|} \Big( \ell_2^2 \frac{\alpha^2}{2} - 2 \ell_1 \ell_2 \cos \varphi \frac{4 \alpha \beta}{\pi^2} + \ell_1^2 \frac{\beta^2}{2} \Big) \\
 & = \frac{\pi^2}{2 |\cP|} \left\langle \widetilde A \begin{pmatrix} \alpha \\ \beta \end{pmatrix}, \begin{pmatrix} \alpha \\ \beta \end{pmatrix} \right\rangle,
\end{split}
\end{align}
where the matrix $\widetilde A$ is given by
\begin{align*}
 \widetilde A = \begin{pmatrix} \ell_2^2 & - \frac{8}{\pi^2} \ell_1 \ell_2 \cos \varphi \\ - \frac{8}{\pi^2} \ell_1 \ell_2 \cos \varphi & \ell_1^2 \end{pmatrix}.
\end{align*}
Furthermore,
\begin{align*}
 \int_\cQ \frac{1}{f} |u|^2 = \frac{|\cP|}{2} (\alpha^2 + \beta^2).
\end{align*}
The matrix $\widetilde A$ has the eigenvalues
\begin{align*}
 \frac{|\cP|^2}{\pi^2} \lambda_\pm,
\end{align*}
and we choose corresponding mutually orthogonal eigenvectors $(\alpha_-,  \beta_-)^\top$ and \linebreak $(\alpha_+, \beta_+)^\top$. Let, moreover, $u_\pm (s, t) = \alpha_\pm \cos (\pi s) + \beta_\pm \cos (\pi t)$ be the versions of~\eqref{eq:cosTest} with coefficients corresponding to the chosen eigenvectors of $\widetilde A$. Then we get
\begin{align}\label{eq:Rayleighmu2first}
 \frac{\int_{\cQ} \langle A \nabla u_-, \nabla u_- \rangle}{\int_\cQ \frac{1}{f} |u_-|^2} = \lambda_-
\end{align}
and 
\begin{align*}
 \frac{\int_{\cQ} \langle A \nabla u_+, \nabla u_+ \rangle}{\int_\cQ \frac{1}{f} |u_+|^2} = \lambda_+.
\end{align*}
Applying the min-max principle we get
\begin{align}\label{eq:firstHalfMu2}
 \mu_2 (\cP) \leq \lambda_-
\end{align}
and, noting that $\int_\cQ \frac{1}{f} u_- u_+ = 0$ due to orthogonality of $(\alpha_-, \beta_-)^\top$ and $(\alpha_+, \beta_+)^\top$,
\begin{align}\label{eq:firstHalfMu3}
 \mu_3 (\cP) \leq \lambda_+.
\end{align}
The latter estimates constitute \eqref{eq:mu3} and parts of \eqref{eq:mu2} and yield in their combination \eqref{eq:mean}.

Secondly, consider the function
\begin{align*}
 u (s, t) = \alpha (s - 1/2) + \beta (t - 1/2), \quad (s, t) \in \cQ,
\end{align*}
where $\alpha, \beta$ are again arbitrary real coefficients. Clearly, $\int_\cQ u = 0 = \int_\cQ \frac{1}{f} u$ holds and $\nabla u (x) = (\alpha, \beta)^\top$. Hence,
\begin{align*}
 \int_\cQ \langle A \nabla u, \nabla u \rangle & = \left\langle A \binom{\alpha}{\beta}, \binom{\alpha}{\beta} \right\rangle.
\end{align*}
On the other hand, 
\begin{align*}
 \int_\cQ \frac{1}{f} |u|^2 & = \frac{|\cP|}{12} (\alpha^2 + \beta^2).
\end{align*}
The eigenvalues of the matrix $A$ are given by
\begin{align*}
 \frac{|\cP|}{12} \eta_\pm.
\end{align*}
We choose an eigenvector $(\alpha_-, \beta_-)^\top$ corresponding to $|\cP| \eta_- / 12$ and let $u_- (s, t) = \alpha_- (s - 1/2) + \beta_- (t - 1/2)$. As in the previous case we get
\begin{align}\label{eq:Rayleighmu2second}
 \frac{\int_{\cQ} \langle A \nabla u_-, \nabla u_- \rangle}{\int_\cQ \frac{1}{f} |u_-|^2} = \eta_-.
\end{align}
From this identity we conclude
\begin{align*}
 \mu_2 (\cP) \leq \eta_-.
\end{align*}
The latter together with \eqref{eq:firstHalfMu2} proves \eqref{eq:mu2}.

Let us now consider the cases of equality. The case where $\lambda_+=\lambda_-$ follows immediately from Remark \ref{remEquality}. Hence, we assume in the rest of the proof that $\lambda_-<\lambda_+$, or equivalently that $\cP$ is not a square.

 Let us first assume that equality holds in \eqref{eq:mu2}. Then $\mu_2 (\cP) = \lambda_-$ or $\mu_2 (\cP) = \eta_-$; in the second case, the function $\alpha_- (s - 1/2) + \beta_- (t - 1/2)$ is an eigenfunction of $u \mapsto - f (\cdot) \div (A \nabla u)$ on $\cQ$ with vanishing co-normal derivative corresponding to $\mu_2 (\cP)$, which is impossible, as the left-hand side of the equation $- f (\cdot) \div (A \nabla u) = \mu_2 (\cP) u$ is constantly zero in this case. Therefore we must have $\mu_2 (\cP) = \lambda_-$, and from the equation \eqref{eq:Rayleighmu2first} we then get that $u_- (s, t) = \alpha_- \cos (\pi s) + \beta_- \cos (\pi t)$ is an eigenfunction of $u \mapsto - f (\cdot) \div (A \nabla u)$ on $\cQ$ with vanishing co-normal derivative corresponding to $\mu_2 (\cP)$, i.e.\
\begin{align*}
 - f (\cdot) \div (A \nabla u_-) = \mu_2 (\cP) u_- \quad \text{in}~\cQ, \qquad \langle A \nabla u_-, \nu \rangle = 0 \quad \text{on}~\partial \cQ.
\end{align*} 
The eigenvalue equation is 
\begin{align*}	
	0 & = - f (s,t) \div (A \nabla u_-)(s,t)-\lambda_\pm u_-(s,t) \\
	& = \alpha_-\left(\frac{\pi^2\ell_2^2}{|\cP|^2}-\lambda_-\right)\cos(\pi s)+\beta_-\left(\frac{\pi^2\ell_1^2}{|\cP|^2}-\lambda_-\right)\cos(\pi t)
\end{align*}
for all $(t,s)\in\cQ$.  Since $u_-$ is not zero, one of $\alpha_-$ and $\beta_-$ is not zero. If none of them are zero,
\[\lambda_-=\frac{\pi^2\ell_2^2}{|\cP|^2}=\frac{\pi^2\ell_1^2}{|\cP|^2},\]
so that $\ell_1=\ell_2$ and, from the explicit formula for $\lambda_-$,  $\cos\varphi=0$. Then, $\cP$ is a square, which contradicts our initial assumption. Therefore either $\alpha_-=0$ or $\beta_-=0$, meaning that either $(1,0)^\top$ or $(0,1)^\top$ is an eigenvector of the matrix $\widetilde A$, corresponding respectively to the eigenvalue $\ell_2^2$ or $\ell_1^2$. Using the explicit formula for $\widetilde A$, this implies $\cos\varphi=0$, so that $\cP$ is a rectangle. It is easy to check, however, that in this case $\mu_2 (\cP) = \lambda_-$ is true.
 
Next, let us assume that equality holds in \eqref{eq:mu3}. Then the variational characterization of Corollary \ref{cor:Rayleigh} yields that there exists a linear combination $u = \gamma u_- + \delta u_+$ of $u_-$ and $u_+$ which is an eigenfunction of $u \mapsto - f (\cdot) \div (A \nabla u)$ on $\cQ$ with vanishing co-normal derivative corresponding to $\lambda_+ = \mu_3 (\cP)$; in particular,
\begin{align*}
 \int_\cQ \langle A \nabla u, \nabla u\rangle = \lambda_+ \int_\cQ \frac{1}{f} |u|^2.
\end{align*}
Since, similarly to \eqref{eq:formComp},
\begin{align*}
 \int_\cQ \langle A \nabla u_-, \nabla u_+\rangle = \frac{\pi^2}{2 |\cP|} \left\langle \widetilde A \binom{\alpha_-}{\beta_-}, \binom{\alpha_+}{\beta_+} \right\rangle = \frac{|\cP|}{2} \lambda_+ \left\langle \binom{\alpha_-}{\beta_-}, \binom{\alpha_+}{\beta_+} \right\rangle = 0
\end{align*}
and $\int_\cQ \frac{1}{f} u_- u_+ = 0$, it follows from the previous two formulas that
\begin{align*}
 \lambda_+ \int_\cQ \frac{1}{f} \left( \gamma^2 |u_-|^2 + \delta^2 |u_+|^2 \right) & = \int_\cQ \langle A \nabla u, \nabla u \rangle \\
 & = \lambda_- \int_\cQ \frac{1}{f} \gamma^2 |u_-|^2 + \lambda_+ \int_\cQ \frac{1}{f} \delta^2 |u_+|^2.
\end{align*}
As $\lambda_- < \lambda_+$ by assumption, it follows $\gamma = 0$, i.e., $u_+ (s, t) = \alpha_+ \cos (\pi s) + \beta_+ \cos (\pi t)$ is an eigenfunction of $u \mapsto - f (\cdot) \div (A \nabla u)$ on $\cQ$ with vanishing co-normal derivative corresponding to $\mu_3 (\cP)$. Repeating the argument from the previous case, we obtain that $\cP$ is a rectangle. Then
\[\mu_3(\cP)=\min\left\{\frac{\pi^2}{\ell_1^2},\frac{4\pi^2}{\ell_2^2}\right\},\]
while
\[\lambda_+=\frac{\pi^2}{\ell_1^2}.\]
Since $\mu_3(\cP)=\lambda_+$ by hypothesis, we necessarily have $\pi^2/\ell_1^2\le4\pi^2/\ell_2^2$, that is $\ell_2\le2\ell_1$.

Let us finally assume that equality holds in \eqref{eq:mean}. Since $\mu_2(\cP)\le\lambda_-$, $\mu_3(\cP)\le \lambda_+$ and 
\[\ \frac{\pi^2}{|\cP|^2} \frac{\ell_1^2 + \ell_2^2}{2}=\frac{\lambda_-+\lambda_+}2,\]
this implies $\mu_2(\cP)=\lambda_-$ and $\mu_3(\cP)=\lambda_+$, so 
that the previous case applies and we obtain the same conclusion.
\end{proof}

\begin{remark}
Using the affine linear trial functions in the second part of the previous proof one can easily derive the additional estimate $\mu_3 (\cP) \leq \eta_+$ for each parallelogram. However, it is easy to see that the estimate \eqref{eq:mu3} is always better, i.e.\ $\lambda_+ < \eta_+$ for each choice of the side lengths $\ell_1, \ell_2$ and the angle $\varphi$. On the other hand, depending on these parameters, one or the other estimate given in \eqref{eq:mu2} may be stronger. This depends on the side lengths and the angle. Roughly speaking, the estimate $\mu_2 \leq \lambda_-$ obtained from cosinoidal trial functions is better as long as $\cP$ is close enough to a square. However, for instance if $\ell_1 = \ell_2 = 1$ and $\varphi = \frac{\pi}{4}$, then 
\begin{align*}
 \lambda_- = \frac{1}{|\cP|^2} \left( \pi^2 - \frac{8}{\sqrt{2}} \right) > \frac{1}{|\cP|^2} \left( 12 - \frac{12}{\sqrt{2}} \right) = \eta_-.
\end{align*}
\end{remark}

\begin{remark}
The estimates in Theorem \ref{thm:parallel} are sharp as they yield the exact eigenvalues in the case of a square and for certain rectangles.
\end{remark}

If $\cP$ is a rhombus, i.e.\ an equilateral parallelogram, then the estimates simplify:

\begin{corollary}
Assume that $\cP$ is a rhombus. If $\ell$ denotes the length of an arbitrary side and $\varphi$ is one of the angles then the estimates
\begin{align*}
\begin{split}
 \mu_2 (\cP) \leq \min & \Bigg\{ \frac{\pi^2}{|\cP|^2} \ell^2 \left( 1 - \frac{8}{\pi^2} |\cos \varphi| \right), \frac{12}{|\cP|^2} \ell^2 \left( 1 - |\cos \varphi| \right) \Bigg\},
\end{split}
\end{align*}
\begin{align*}
 \mu_3 (\cP) \leq \frac{\pi^2}{|\cP|^2} \ell^2 \left( 1 + \frac{8}{\pi^2} |\cos \varphi| \right),
\end{align*}
and
\begin{align*}
 \frac{\mu_2 (\cP) + \mu_3 (\cP)}{2} \leq \frac{\pi^2}{|\cP|^2} \ell^2.
\end{align*}
In each estimate, equality holds if and only if $\cP$ is a square.
\end{corollary}

The following may be compared with \cite[Theorem 3.5]{LS09}, where triangular domains are considered;  note that the quantity on the left-hand side of \eqref{eq:scalingInv} below is scaling-invariant.

\begin{corollary}
Let $\cP$ be any parallelogram and let $S := \sqrt{\ell_1^2 + \ell_2^2}$. Then
\begin{align}\label{eq:scalingInv}
 \frac{\mu_2 (\cP) + \mu_3 (\cP)}{2} \frac{|\cP|^2}{S^2} \leq \frac{\pi^2}{2},
\end{align}
with equality if and only if $\cP$ is a rectangle with $\ell_2 \leq 2 \ell_1$, where $\ell_1 \leq \ell_2$ are its side lengths.
\end{corollary}

As a further indication of sharpness, we deduce an isoperimetric inequality from the above spectral estimates. More specifically, we prove that among all parallelograms of fixed perimeter, the square is the only maximizer of $\mu_2 (\cP)$. 

As mentioned in the introduction, the same result was proved recently in \cite[Proposition~4.3]{HLL22}, using a different technique. It complements \cite[Theorem 3.1]{LS09}, which shows that among all triangles $\cT$ of fixed perimeter $L(\cT)$, the equilateral one is the only maximizer of $\mu_2(\cT)$ and gives the same value of $\mu_2 (\cT) L (\cT)^2$ as the square. This partially answers the question, raised by R.~Laugesen in \cite[p. 405, Problem 3]{WBAL09}, of whether those two shapes maximize $\mu_2 (\Omega) L (\Omega)^2$ among all bounded convex domains $\Omega$.

\begin{theorem}\label{thm:parallelIsoperimetric}
Let $\cP \subset \R^2$ be any parallelogram and let $L (\cP)$ denote its perimeter. Then
\begin{align*}
 \mu_2 (\cP) L (\cP)^2 \leq 16 \pi^2.
\end{align*}
Equality holds if and only if $\cP$ is a square.
\end{theorem}

\begin{proof}
Since the quantity $\mu_2 (\cP) L (\cP)^2$ is scaling-invariant, after possible rotation, reflection and rescaling we may assume that $\cP$ is spanned by the vectors $(a, b)^\top$ and $(1, 0)^\top$, where $a \geq 0, b > 0$, and $a^2 + b^2 \leq 1$. In this case, the estimates in Theorem \ref{thm:parallel} yield
\begin{align}\label{eq:first}
  \mu_2 (\cP) L (\cP)^2 \leq \frac{2 \pi^2}{b^2} (1 + \sqrt{a^2 + b^2})^2 \left( a^2 + b^2 + 1 - \sqrt{\big(a^2 + b^2 - 1 \big)^2 + \frac{256}{\pi^4} a^2} \right)
\end{align}
and
\begin{align}\label{eq:second}
 \mu_2 (\cP) L (\cP)^2 \leq \frac{24}{b^2} (1 + \sqrt{a^2 + b^2})^2 \left( a^2 + b^2 + 1 - \sqrt{\big(a^2 + b^2 - 1 \big)^2 + 4 a^2} \right),
\end{align}
where we have used $\ell_1 = \sqrt{a^2 + b^2}$, $\ell_2 = 1$, $\cos \varphi = a / \ell_1$, $|\cP| = b$, and $L (\cP) = 2 (1 + \sqrt{a^2 + b^2})$.

We distinguish three (not mutually distinct) cases. Firstly, consider the case 
\begin{align}\label{eq:case1}
 b > \sqrt{1 - \frac{64}{\pi^4}} \approx 0.59.
\end{align}
By this assumption we have
\begin{align*}
 (a^2 + b^2 - 1)^2 + \frac{256}{\pi^4} a^2 & = a^4 + 2 a^2 \Big( b^2 - 1 + \frac{128}{\pi^4} \Big) + (b^2 - 1)^2 \\
 & = a^4 + 2 a^2 \Big( 2 b^2 - 1 + \frac{128}{\pi^4} - b^2 \Big) + (b^2 - 1)^2 \\
 & \geq a^4 + 2 a^2 (1 - b^2) + (1 - b^2)^2 \\
 & = (a^2 + 1 - b^2)^2,
\end{align*}
where equality is only possible if $a = 0$. Hence \eqref{eq:first} implies
\begin{align}\label{eq:either}
 \mu_2 (\cP) L (\cP)^2 \leq 4 \pi^2 (1 + \sqrt{a^2 + b^2})^2 \leq 16 \pi^2,
\end{align}
where we have used $a^2 + b^2 \leq 1$; equality in \eqref{eq:either} holds if and only if $a = 0$ and $a^2 + b^2 = 1$, that is, if $(a, b) = (0, 1)$, which is the case that $\cP$ is a square. 

Secondly, let
\begin{align}\label{eq:case2}
 \sqrt{a^2 + b^2} < \frac{\pi}{\sqrt{3}} - 1 \approx 0.81.
\end{align}
Note that 
\begin{align*}
 (a^2 + b^2 - 1)^2 + 4 a^2 & = a^4 + 2 a^2 (b^2 - 1) + (b^2 - 1)^2 + 4 a^2 \\
 & = a^4 + 2 a^2 (b^2 + 1) + (b^2 - 1)^2 \\
 & \geq a^4 + 2 a^2 (1 - b^2) + (1 - b^2)^2 \\
 & = (a^2 + 1 - b^2)^2.
\end{align*}
Then \eqref{eq:second} together with \eqref{eq:case2} gives
\begin{align*}
 \mu_2 (\cP) L (\cP)^2 \leq 48 (1 + \sqrt{a^2 + b^2})^2 < 48 \frac{\pi^2}{3} = 16 \pi^2.
\end{align*}

In the third and final case we assume
\begin{align}\label{eq:case3}
 a > \frac{12}{\pi^2} + 1 - \frac{\pi}{\sqrt{3}} \approx 0.40 \quad \text{and} \quad \sqrt{a^2 + b^2} \geq \frac{\pi}{\sqrt{3}} - 1.
\end{align}
Writing $r = \ell_1 = \sqrt{a^2 + b^2} \leq 1$, note first that
\begin{align*}
 0 & \geq 4 (a - r) (1 - r)^2 = 4 (r - a) \big(r + a + r - a - r^2 - 1 \big) \\
 & = 4 \big( r^2 - a^2 + (r - a)^2 - (r^2 + 1) (r - a) \big) \\
 & = (r^2 + 1)^2 - (r^2 - 1)^2 + 4 (r - a)^2 - 4 (r^2 + 1) (r - a) - 4 a^2,
\end{align*}
and thus
\begin{align*}
 (r^2 + 1)^2 - 4 (r^2 + 1) (r - a) + 4 (r - a)^2 \leq (r^2 - 1)^2 + 4 a^2.
\end{align*}
As the left-hand side equals $(r^2 + 1 - 2 (r - a))^2$ and both sides are positive, we get
\begin{align*}
 r^2 + 1 - 2 (r - a) \leq \sqrt{(r^2 - 1)^2 + 4 a^2}
\end{align*}
or, equivalently,
\begin{align*}
 a^2 + b^2 + 1 - \sqrt{(a^2 + b^2 - 1)^2 + 4 a^2} \leq 2 \big(\sqrt{a^2 + b^2} - a \big).
\end{align*}
From this and \eqref{eq:second} we conclude
\begin{align*}
 \mu_2 (\cP) L (\cP)^2 & \leq \frac{48}{(\sqrt{a^2 + b^2} - a) (\sqrt{a^2 + b^2} + a)} (1 + \sqrt{a^2 + b^2})^2 \big(\sqrt{a^2 + b^2} - a \big) \\
 & = 48 \frac{(1 + \sqrt{a^2 + b^2})^2}{\sqrt{a^2  + b^2} + a}. 
\end{align*}
Applying the assumption \eqref{eq:case3} and $a^2 + b^2 \leq 1$ to the latter estimate yields
\begin{align*}
 \mu_2 (\cP) L (\cP)^2 & < \frac{192}{\frac{12}{\pi^2}} = 16 \pi^2.
\end{align*}
Since one easily observes that each choice of $(a, b)^\top$ with $a \geq 0, b > 0$ and $a^2 + b^2 \leq 1$ satisfies one of the assumptions \eqref{eq:case1}, \eqref{eq:case2} or \eqref{eq:case3}, the proof is complete.
\end{proof}

As a direct consequence, among all parallelograms of fixed area, the square maximizes $\mu_2 (\cP)$.

\begin{corollary}
Let $\cP \subset \R^2$ be any parallelogram and let $|\cP|$ denote its area. Then
\begin{align*}
 \mu_2 (\cP) |\cP| \leq \pi^2.
\end{align*}
Equality holds if and only if $\cP$ is a square.
\end{corollary}

\begin{proof}
Let now $\cQ$ denote a square with the same perimeter as $\cP$. Then, by Theorem~\ref{thm:parallelIsoperimetric},
$\mu_2 (\cP) \leq \mu_2 (\cQ)$. As $|\cQ| \geq |\cP|$,
\begin{align*}
 \mu_2 (\cP) |\cP| \leq \mu_2 (\cQ) |\cQ| = \pi^2.
\end{align*}
Moreover, in all these estimates, equality holds if and only if $\cP$ is a square.
\end{proof}

\section{Domains of constant width}\label{sec:constantWidth}

In this section we use the approach of tranforming the Laplacian on a domain into an elliptic operator with a weight function on a square discussed in Section~\ref{sec:prel} to obtain spectral estimates for another class of domains. We make the following assumption.

\begin{assumption}\label{ass}
$\Omega \subset \R^2$ has the form
\begin{align*}
 \Omega = \left\{ (x, y)^\top: 0 < x < \ell, g (x) < y < h (x) \right\},
\end{align*}
where $g, h \in C^1 ([0, \ell])$ are real-valued functions such that $d (x) := h (x) - g (x)$ is uniformly positive on $[0, \ell]$.
\end{assumption}

To map a domain as in Assumption~\ref{ass} onto a square, consider the mapping
\begin{align*}
 \Phi (x, y) = \begin{pmatrix} x / \ell \\ \frac{y - g (x)}{d (x)} \end{pmatrix}, \quad (x, y)^\top \in \Omega.
\end{align*}
It maps $\Omega$ one-to-one onto the square $\cQ := (0, 1)^2$. The Jacobian of $\Phi$ is given by
\begin{align*}
 (D \Phi) (x, y) = \begin{pmatrix} 1/\ell & 0 \\ \frac{- g' (x) d (x) - (y - g (x)) d' (x)}{d (x)^2} & 1 / d (x) \end{pmatrix}.
\end{align*}
Note that $\Phi$ is a $C^1$-diffeomorphism and that all first-order partial derivatives of $\Phi$ and $\Phi^{-1}$ are bounded due to the assumption that $d$ is uniformly positive and bounded. Furthermore,
\begin{align*}
 \det (D \Phi) (x, y) = \frac{1}{\ell d (x)}.
\end{align*}
Then the matrix $A (s, t)$ defined in \eqref{eq:A} is given by
\begin{align*}
 (A \circ \Phi) (x, y) = \begin{pmatrix} \frac{d (x)}{\ell} & \frac{- g' (x) d (x) - (y - g (x)) d' (x)}{d (x)} \\ \frac{- g' (x) d (x) - (y - g (x)) d' (x)}{d (x)} & \frac{\ell}{d (x)} + \frac{\big(g' (x) d (x) + (y - g (x)) d' (x) \big)^2}{d (x)^3} \ell \end{pmatrix}
\end{align*}
for $(x, y)^\top \in \Omega$ and, hence,
\begin{align*}
 A (s, t) = \begin{pmatrix} \frac{d (\ell s)}{\ell} & - g' (\ell s) - t d' (\ell s) \\ - g' (\ell s) - t d' (\ell s) & \frac{\ell}{d (\ell s)} \Big( 1 + \big(g' (\ell s)  + t d' (\ell s) \big)^2 \Big) \end{pmatrix}
\end{align*}
for $(s, t)^\top \in (0, 1)^2$. This transformation can be used to obtain estimates for the eigenvalues of the Neumann Laplacian on $\Omega$. We will now illustrate this at the example of domains of constant width; cf.\ Figure \ref{fig:constWidth}. 
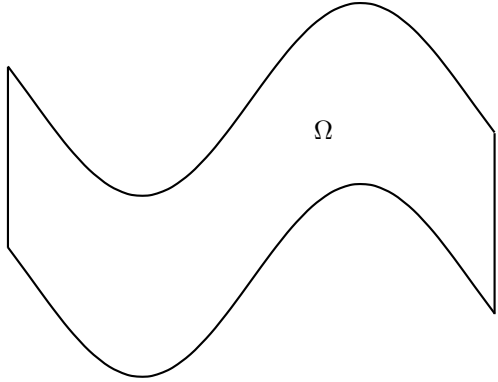
\begin{figure}[h]
    
\begin{tikzpicture}[scale=0.8]
% labels
 \draw[white] (1.5,2.5) circle(0.00) node[left,black]{$\Omega$};
% domain
 \draw[scale=1,domain=-4:4,smooth,variable=\x,thick] plot ({\x},{1.6*sin(50*\x)});
 \draw[scale=1,domain=-4:4,smooth,variable=\x,thick] plot ({\x},{1.6*sin(50*\x)+3});
 \draw[thick] (-4,0.54) -- (-4,3.54);
 \draw[thick] (4,-0.56) -- (4,2.44);
\end{tikzpicture} 
\caption{A domain of constant width.}
\label{fig:constWidth}
\end{figure}
Let us point out that all these domains are non-convex, except for the rectangle.

\begin{theorem}\label{thm:constantWidth}
Suppose Assumption \ref{ass} holds and $d$ is constant. Define
\begin{align*}
 \lambda_\pm & = \frac{\pi^2}{2 \ell d} \Bigg( \frac{d}{\ell} + \frac{1}{d} \int_0^\ell \big( 1 + g' (x)^2 \big) \textup{d} x \\
 & \qquad \pm \sqrt{\bigg( \frac{d}{\ell} - \frac{1}{d} \int_0^\ell \big( 1 + g' (x)^2 \big) \textup{d} x \bigg)^2 + \frac{64}{\pi^2 \ell^2} \bigg( \int_0^\ell g' (x) \sin (\pi x / \ell) \textup{d} x \bigg)^2} \Bigg).
\end{align*}
Then
\begin{align}\label{eq:lambda}
 \mu_2 (\Omega) \leq \lambda_- \quad \text{and} \quad \mu_3 (\Omega) \leq \lambda_+.
\end{align}
In particular,
\begin{align}\label{eq:mu2ConstantWidth}
 \mu_2 (\Omega) \leq \min \left\{ \frac{\pi^2}{\ell^2}, \frac{\pi^2}{d^2 \ell} \int_0^\ell \big( 1 + g' (x)^2 \big) \dd x \right\}
\end{align}
and
\begin{align}\label{eq:arithmeticMean}
 \frac{\mu_2 (\Omega) + \mu_3 (\Omega)}{2} \leq \frac{\pi^2}{2} \bigg( \frac{1}{\ell^2} + \frac{1}{\ell d^2} \int_0^\ell \big(1 + g' (x)^2 \big) \textup{d} x \bigg)
\end{align}
hold. Moreover, equality in the first estimate in \eqref{eq:lambda} or in \eqref{eq:mu2ConstantWidth} holds if, and only if, $\Omega$ is a rectangle.
\end{theorem}

\begin{proof}
As $d$ is constant, in the notation of Lemma \ref{lem:trafo} also $1/f = \ell d$ is constant. Thus the function
\begin{align}\label{eq:eigenfunction}
 u (s, t) = \alpha \cos (\pi s) + \beta \cos (\pi t),
\end{align}
where $\alpha, \beta$ are arbitrary real constants, satisfies $\int_\cQ \frac{1}{f} u = 0$, and
\begin{align*}
 \int_\cQ \frac{1}{f} |u|^2 = \frac{\ell d}{2} (\alpha^2 + \beta^2).
\end{align*}
Moreover,
\begin{align*}
 \int_\cQ \langle A \nabla u, \nabla u \rangle & = \pi^2 \Bigg( \alpha^2 \frac{d}{\ell} \int_0^1 \sin^2 (\pi s) \dd s - 2 \alpha \beta \int_0^1 g' (\ell s) \sin (\pi s) \dd s \int_0^1 \sin (\pi t) \dd t \\
 & \quad + \beta^2 \frac{\ell}{d} \int_0^1 \big( 1 + g' (\ell s)^2 \big) \dd s \int_0^1 \sin^2 (\pi t) \dd t \Bigg) \\
 & = \pi^2 \Bigg( \alpha^2 \frac{d}{2 \ell} - \frac{4\alpha \beta}{\pi\ell} \int_0^\ell g' (x) \sin (\pi x/\ell) \dd x + \frac{\beta^2}{2 d} \int_0^\ell \big(1 + g' (x)^2 \big) \dd x \Bigg) \\
 & = \left\langle M \binom{\alpha}{\beta}, \binom{\alpha}{\beta} \right\rangle,
\end{align*}
where
\begin{align*}
 M = \frac{\pi^2}{2} \begin{pmatrix}
  \frac{d}{\ell} & - \frac{4}{\pi\ell} \int_0^\ell g' (x) \sin (\pi x/\ell) \dd x \\
  - \frac{4}{\pi\ell} \int_0^\ell g' (x) \sin (\pi x/\ell) \dd x & \frac{1}{d} \int_0^\ell \big(1 + g' (x)^2 \big) \dd x
 \end{pmatrix}.
\end{align*}
The matrix $M$ has eigenvalues $\frac{\ell d}{2} \lambda_\pm$, with $\lambda_\pm$ given in the theorem. Choosing $(\alpha, \beta)^\top$ equal to the corresponding eigenvectors and applying Corollary~\ref{cor:Rayleigh} yields the estimates \eqref{eq:lambda}. On the other hand, choosing $(\alpha, \beta)$ equal to the standard basis vectors gives
\begin{align*}
 \mu_2 (\Omega) \leq \pi^2 \frac{d}{2 \ell} \frac{2}{\ell d}
\end{align*}
and
\begin{align*}
 \mu_2 (\Omega) \leq \frac{\pi^2}{2 d} \int_0^\ell \big(1 + g' (x)^2 \big) \dd x \frac{2}{\ell d},
\end{align*}
respectively, which yields \eqref{eq:mu2ConstantWidth}. Finally, the estimate \eqref{eq:arithmeticMean} is a direct consequence of \eqref{eq:lambda}.

It remains to discuss the cases of equality. First we show  that $\mu_2 (\Omega) = \lambda_-$ is only possible if $\Omega$ is a rectangle. Assume for a contradiction that this equality holds for a non-rectangle $\Omega$. Then there exists a nonempty open interval $I \subset (0, \ell)$ on which $g'$ is nowhere vanishing. Moreover, there exists a coefficient pair $(\alpha, \beta) \neq (0, 0)$ such that the function $u$ in \eqref{eq:eigenfunction} is an eigenfunction of $- \frac{1}{\ell d} \div (A (\cdot) \nabla u)$ on $\cQ = (0, 1)^2$ with vanishing co-normal derivative; for $(s, 0) \in \partial \mathcal Q$ the latter read 
\begin{align*}
 \begin{pmatrix} \frac{d}{\ell} & - g' (\ell s) \\ - g' (\ell s) & \frac{\ell}{d} \left( 1 + g' (\ell s)^2 \right) \end{pmatrix} \nabla u \cdot \binom{0}{-1} = 0
\end{align*}
and can be written
\begin{align*}
 \pi g' (\ell s) \alpha \sin (\pi s) = 0.
\end{align*}
Choosing $\ell s \in I$ this implies $\alpha = 0$, i.e.\ $u (s, t) = \beta \cos (\pi t)$. 
Let us now show that the eigenvalue equation,  of which $u$ is a distributional solution, implies that $g$ is linear. To simplify notation, we temporarily set
\begin{align*}
			a(s):=&-\frac{1}{\ell d}g'(\ell s);\\
			b(s):=&\frac{1}{d^2}\left(1+g'(\ell s)^2\right).
\end{align*}  
Then, for any $\varphi\in C_c^\infty(\mathcal Q)$, 
\[\beta\int_{\mathcal Q}\left(\pi\sin(\pi t)\left(a(s)\partial_s\varphi(s,t)+b(s)\partial_t \varphi(s,t)\right)+\mu_2(\Omega)\cos(\pi t)\varphi(s,t)\right)\dd(s,t)=0.\]
Integrating by parts,
\[\int_0^1\sin(\pi t)\partial_t\varphi(s,t) \dd t=-\pi\int_0^1\cos(\pi t)\varphi(s,t) \dd t=0.\]
Using Fubini's theorem and the previous equality, we find that for any $\varphi\in C^\infty_c(\mathcal Q)$,
\[\beta\int_{\mathcal Q}\left(\pi\sin(\pi t)a(s)\partial_s\varphi(s,t)+\cos(\pi t)\left(\mu_2(\Omega)-\pi^2b(s)\right)\varphi(s,t)\right)\dd(s,t)=0.\]
We now apply the previous formula to $\varphi(s,t)=\xi(s)\chi_n(t)$, where $\xi$ is an arbitrary function in $C_c^\infty((0,1))$ and $(\chi_n)$ is a sequence in $C_c^\infty((0,1))$ converging to the $\delta$-distribution centered at $t=1/2$. Taking $n\to\infty$, we get that 
\[\beta\int_0^1a(s)\xi'(s)\dd s=0\] 
for any $\xi\in C^\infty_c((0,1))$. Since $\beta$ cannot be $0$, this means that the distributional derivative of the function $a$ is $0$, which implies that $a$ is a constant. Therefore $g'$ is a constant and $g$ is linear.
On the other hand, on the boundary lines $s = 0$ and $s = 1$ the boundary condition gets
\begin{align*}
 \mp \pi g' (\ell s) \beta \sin (\pi t) = 0, 
\end{align*}
$t \in (0, 1)$ and, thus, $g' (0) = 0 = g' (\ell)$; but then the linear function $g$ is constant and $\Omega$ a rectangle, another contradiction.

Now assume that equality holds in \eqref{eq:mu2ConstantWidth} and that $\Omega$ is not a rectangle. Then either $\cos (\pi s)$ or $\cos (\pi t)$ is an eigenfunction of $- \frac{1}{\ell d} \div (A (\cdot) \nabla u)$ with Neumann boundary conditions and a reasoning analogous to the above one leads to a contradiction.

To complete the proof of the theorem, it remains to note that if $\Omega$ is a rectangle, i.e.\ $g$ is constant, then the bounds for $\mu_2 (\Omega)$ in both the first estimate in \eqref{eq:lambda} and \eqref{eq:mu2ConstantWidth} read $\min \{\pi^2/\ell^2, \pi^2/d^2\}$, being equal to the lowest positive eigenvalue of the rectangle of length $\ell$ and width $d$.
\end{proof}

\begin{example}
Although all the domains that are admissible in the theorem have area $\ell d$, the estimate for $\mu_2 (\Omega)$ given in the theorem is not necessarily below $\frac{\pi^2}{\ell d}$, the first eigenvalue of the square of the same area. Consider, for instance, the domain given by
\begin{align*}
 \Omega_\eps = \left\{ (x, y)^\top : 0 < x < \pi, \sin (x) - \Big( \frac{\pi}{2} + \eps \Big) < y < \sin (x) + \Big( \frac{\pi}{2} + \eps \Big) \right\}
\end{align*}
for sufficiently small $\eps > 0$. In this case, $d = \pi + 2 \eps$ and $\ell = \pi$. Moreover, note that
\begin{align*}
 \int_0^\ell g' (x) \sin (\pi x / \ell) \textup{d} x = 0
\end{align*}
in this case. Therefore Theorem \ref{thm:constantWidth} yields
\begin{align*}
 \mu_2 (\Omega_\eps) & \leq \frac{\pi^2}{2 \ell d} \Bigg( \frac{d}{\ell} + \frac{1}{d} \int_0^\ell \big(1 + g' (x)^2 \big) \textup{d} x - \bigg|\frac{d}{\ell} - \frac{1}{d} \int_0^\ell \big(1 + g' (x)^2 \big) \textup{d} x \bigg| \Bigg) \\
 & = \frac{\pi^2}{2 \ell d} \Bigg( \frac{d}{\ell} + \frac{1}{d} \frac{3 \pi}{2} - \bigg|\frac{d}{\ell} - \frac{1}{d} \frac{3 \pi}{2} \bigg| \Bigg).
\end{align*}
Note that for sufficiently small $\eps$ the term inside the modulus is negative and, hence, the estimate yields
\begin{align*}
 \mu_2 (\Omega_\eps) \leq \frac{\pi^2}{\ell^2}.
\end{align*}
Since $\ell < d$, the latter is larger than $\frac{\pi^2}{\ell d}$, the first eigenvalue of the square with the same area as $\Omega$.
\end{example}

Despite the previous example, the estimate \eqref{eq:mu2ConstantWidth} in Theorem \ref{thm:constantWidth} has the following immediate implication.

\begin{corollary}\label{cor:Long}
Suppose that $\Omega$ is a domain of the form as given in Assumption~\ref{ass} with constant $d$ and with $\ell \geq d$. Then
\begin{align}\label{eq:jetztAberLos}
 \mu_2 (\Omega) \leq \frac{\pi^2}{\ell^2}.
\end{align}
Equality holds if and only if $\Omega$ is a rectangle.
\end{corollary}

Finally, we point out that this corollary yields the inequality of Conjecture \ref{conj} when $\Omega$ is a domain, of the previous form, close to a sufficiently elongated rectangle. Let us give a more precise and quantitative statement.

\begin{proposition}\label{prop:NonConvexClass} For any $\rho\in(0,1)$, let $\mathcal A_\rho$ denote the set of domains satisfying Assumption \ref{ass}, with constant $d$, and for which, in addition,\[\frac{d}{\ell}\le \rho \mbox{\hspace{10pt} and \hspace{10pt} }\|g'\|_\infty\le M_\rho,\]
with
\[M_\rho:=\sqrt{(2-\rho)^2-1}.\]
Then, for any $\Omega\in\mathcal A_\rho$,
\begin{equation*}
 	L(\Omega)^2\mu_2(\Omega)<16\pi^2.
\end{equation*}
\end{proposition}

\begin{remark} For any $\rho\in(0,1)$, $\mathcal A_\rho$ mostly contains non-convex domains. Indeed, the only convex domains in $\mathcal A_\rho$ are rectangles.
\end{remark}

\begin{proof}[Proof of Proposition \ref{prop:NonConvexClass}] Let $\Omega$ be a set in $\mathcal A_\rho$. From Corollary \ref{cor:Long}, it follows that
\begin{align*}
L(\Omega)^2\mu_2(\Omega)&\le \frac{\pi^2}{\ell^2}\left(2d+2\int_0^\ell\sqrt{1+g'(x)^2}\,\textup{d}x\right)^2\\
&\le4\pi^2\left(\frac{d}{\ell}+\sqrt{1+\|g'\|_\infty^2}\right)^2\\
&\le 4\pi^2\left(\rho+\sqrt{1+M_\rho^2}\right)^2=16\pi^2.
\end{align*}
To obtain the strict inequality, we recall that equality in \eqref{eq:jetztAberLos} implies that $\Omega$ is a rectangle, in which case 
\[L(\Omega)^2\mu_2(\Omega)<16\pi^2,\]
since $\Omega$ is not a square. 
\end{proof}

\section{Perturbation approach}\label{sec:perturbation}

Let us conclude with some remarks concerning the behavior of the shape functional $\Omega\mapsto \mu_{2}(\Omega)L(\Omega)^2$ when $\Omega$ is a small perturbation of the unit square $\cQ=(0,1)^2$. We use here a Hadamard-type formula for the shape derivative of a Neumann eigenfunction (see \cite[Sec. 2.5.3]{H06} for a general discussion and \cite[p.\ 1596, Eq.\ (3.12)]{AH11} for the specific formula).  We are not attempting a full justification of its validity. Our goal is merely to check formally that we can find a suitable small perturbation of $\cQ$ into a non-convex domain $\Omega$ such that $\mu_2(\Omega)L(\Omega)^2>16\pi^2$ and we are therefore not overly concerning ourselves with regularity assumptions. 

In complement to this discussion, we recall that \cite {BBG16,HLL22,P08} provide sequences of non-convex domains along which the functional diverges to $+\infty$, as described in more detail in the introduction.

In general, we can deform $\cQ$ in the following way. We fix a smooth vector field \newline $\chi:\R^2 \to\R^2$ with compact support, and define the mapping $\Phi_t(x, y)= (x, y) + t \chi(x, y)$, depending on a real parameter $t$. It is easily checked that $\Phi_t$ is a $C^\infty$-diffeomorphism for $|t|$ small enough.  We then set $\Omega_t:=\Phi_t(\cQ)$ and $L(t):=L(\Omega_t)$. To avoid regularity issues, we assume that $\chi$ vanishes near the corners of $\cQ$. We have therefore reduced the problem to studying $F(t):=\mu_2(\Omega_t)L(t)^2$ for $t$ close to $0$.  

We first note that, according to classical differential geometry,
\[L'(0)=\int_{\partial\cQ}h(\chi\cdot\nu),\]
where $h$ is the curvature of $\partial\cQ$ and $\nu$ the unit normal vector to $\partial \cQ$, pointing outwards. Since $\partial \cQ$ is straight in the support of $\chi$, 
\begin{equation}\label{eq:perimeter}
	L'(0)=0.
\end{equation}

When writing the Hadamard formula, we have to account for the fact that $\mu_2(\cQ)=\mu_3(\cQ)=\pi^2$ is a double eigenvalue. We denote it by $\mu$ and recall that the functions
\begin{align*}	
	u_1(x,y):=&\sqrt2\cos(\pi x);\\
	u_2(x,y):=&\sqrt2\cos(\pi y);
\end{align*}
form an orthonormal basis of the associated eigenspace. Then, we can find two differentiable (indeed, real-analytic) functions  $t\mapsto\mu_1(t),\mu_2(t)$ satisfying the following.
\begin{enumerate}
\item For $|t|$ small enough, $\{\mu_1(t),\mu_2(t)\}=\{\mu_2(\Omega_t),\mu_3(\Omega_t)\}$ (note that the labeling of $\mu_1(t),\mu_2(t)$ does not necessarily coincide with their order in the Neumann spectrum of $\Omega_t$).
 \item The derivatives $\mu_1'(0)$ and $\mu_2'(0)$ are the eigenvalues of the $2\times2$ matrix
\[\left(
\begin{array}{cc}
	\int_{\partial\cQ}\left(|\nabla u_1|^2-\mu u_1^2\right)(\chi\cdot\nu)&\int_{\partial\cQ}\left(\nabla u_1\cdot\nabla u_2-\mu u_1u_2\right)(\chi\cdot\nu)\\
	\int_{\partial\cQ}\left(\nabla u_1\cdot\nabla u_2-\mu u_1u_2\right)(\chi\cdot\nu)&\int_{\partial\cQ}\left(|\nabla u_2|^2-\mu u_2^2\right)(\chi\cdot\nu)
\end{array}
\right),\]
which we denote by $M$.
\end{enumerate}
To carry on with our analysis, we write 
\[\partial\cQ=\Gamma_1\cup\Gamma_2\cup\Gamma_3\cup\Gamma_4,\]
with $\Gamma_1=[0,1]\times\{0\}$, $\Gamma_2=\{1\}\times[0,1]$, $\Gamma_3=[0,1]\times\{1\}$ and $\Gamma_4=\{0\}\times[0,1]$,
and we start imposing additional conditions on $\chi$. First, we assume that the support of $\chi$ intersects only one side of $\cQ$, say $\Gamma_1$, and that we have, on $\Gamma_1$, 
\[\chi(x,0)=(0,-f(x)),\]
with $f$ a non-negative smooth function supported in $(0,1)$, symmetric with respect to the midpoint $x=1/2$. 
 Using these hypotheses, and the explicit formulas for $u_1$ and $u_2$, we find
\[M=2\pi^2\left(
\begin{array}{cc}
	-\int_0^1\cos(2\pi x)f(x)\,dx&0\\
	0&-\int_0^1f(x)\,dx
\end{array}
\right).\]
If we make the additional assumption that $f$ is not identically $0$ and is supported in $(0,1/4)\cup(3/4,1)$, we find 
\begin{equation}\label{eq:matrix}
M=\left(
\begin{array}{cc}
	-\alpha_1&0\\
	0&-\alpha_2
\end{array}
\right),
\end{equation}
with $\alpha_2>\alpha_1>0$. Up to relabeling the functions $t\mapsto \mu_1(t),\ t \mapsto \mu_2(t)$, we can assume that $\mu_1'(0)=-\alpha_1$ and $\mu_2'(0)=-\alpha_2$.

Under the previous hypotheses on $\chi$, the above computations imply that the function $t\mapsto F(t)$ has a left derivative at $0$, given by
\[F'_-(0)=\mu_1'(0)L(0)^2+2\mu_1(0)L'(0)L(0)=-16\alpha_1<0.\]
Thus, we have $F(t)>F(0)=16\pi^2$ for $t$ negative and close enough to $0$. Since the vector field $\chi$, by construction, points outwards of $\cQ$, the corresponding deformation pushes the side $\Gamma_1$ inwards, making the domain $\Omega_t$ slightly non-convex.

\section*{Acknowledgments}

C.L.\ acknowledges support from the INdAM GNAMPA Project  \emph{Operatori differenziali e integrali in geometria spettrale} (CUP\_E53C22001930001). J.R.\ acknowledges support by the grant no.\ 2022-03342 of the Swedish Research Council (VR). The authors are grateful to Pier Domenico Lamberti and Richard S.~Laugesen for help with the references and to Antoine Henrot for helpful suggestions.

% *******************************************************************


\begin{thebibliography}{99}
% *******************************************************************

\bibitem{AH11} P.\,R.\,S.~Antunes and A.~Henrot, {\it On the range of the first two Dirichlet and Neumann eigenvalues of the Laplacian}, Proc.\ R.\ Soc.\ Lond.\ Ser.\ A Math.\ Phys.\ Eng.\ Sci.~467 (2011), 1577--1603.

\bibitem{BBG16} M.~van den Berg, D.~Bucur, and K.~Gittins, {\it Maximising Neumann eigenvalues on rectangles}, Bull.\ Lond.\ Math.\ Soc.\ 48 (2016), no. 5, 877--894.

\bibitem{BLLdC06} V.\,I.~Burenkov, P.\,D.~Lamberti, and M.~Lanza de Cristoforis, {\it Spectral stability of non-negative selfadjoint operators} (Russian), Sovrem.\ Mat.\ Fundam.\ Napravl.\ 15 (2006), 76--111. English translation in: Journal of Mathematical Sciences 149 (2008), 1417--1452.

\bibitem{GU2019} V.~Gol'dshtein and A.~Ukhlov, {\it Composition operators on Sobolev spaces and Neumann eigenvalues}, Complex Anal.\ Oper.\ Theory 13 (2019), no. 6, 2781--2798.

\bibitem{GPU2021} V.~Gol'dshtein, V.~Pchelintsev, and A.~Ukhlov {\it Spectral stability estimates of Neumann divergence form elliptic operators}, Math.\ Rep.\ (Bucur.) 73 (2021), no. 1-2, 132--147.

\bibitem{H06} A.~Henrot, Extremum Problems for Eigenvalues of Elliptic Operators, Frontiers in Mathematics, Birkhäuser Basel, 2006

\bibitem{HLL22}  A.~Henrot, A.~Lemenant, and I.~Lucardesi, {\it An isoperimetric problem with two distinct solutions}, Trans.\ Amer.\ Math.\ Soc., to appear, arXiv:2210.17225.

\bibitem{LS09} R.\,S.\ Laugesen and B.\,A.\ Siudeja, {\it Maximizing Neumann fundamental tones of triangles}, J.\ Math.\ Phys.\ 50 (2009), no. 11, 112903, 18 pp.
 
\bibitem{McL} W.~McLean, Strongly Elliptic Systems and Boundary Integral Equations, Cambridge University Press, Cambridge, 2000.
 
\bibitem{P08} M.~Pang, {\it Stability and approximations of eigenvalues and eigenfunctions for the Neumann Laplacian. II}, J.\ Math.\ Anal.\ Appl.\ 345 (2008), no. 1, 485--499.

\bibitem{R18}  A.~Raiko, {\it Comparison of the first positive Neumann eigenvalues for rectangles and special parallelograms}, preprint, arXiv:1810.07025.

\bibitem{RS1}
M. \ Reed and B.\ Simon, 
Methods of modern mathematical physics. I. Functional Analysis. Revised and Enlarged Edition, Academic Press, 1980.

\bibitem{S54} G.~Szeg\H{o}, {\it Inequalities for certain eigenvalues of a membrane of given area}, J.\ Rational Mech.\ Anal.~3 (1954), 343--356.

\bibitem{WBAL09} T.~Weidl, R.~D.~Benguria, M.~S.~Ashbaugh, and R.~S.~Laugesen, {\it Low Eigenvalues of Laplace and Schrödinger Operators}, Oberwolfach Rep.~6 (2009), no. 1, 355--428.

\bibitem{W56} H.~F.~Weinberger, {\it An isoperimetric inequality for the $n$-dimensional free membrane problem}, J.\ Rational Mech.\ Anal.~5 (1956), 633--636.


\end{thebibliography}
\end{document}